\documentclass[11 pt]{amsart}

\usepackage[margin=2.5cm]{geometry}
\usepackage{amsfonts,graphicx,enumerate}
\usepackage{amssymb}
\usepackage{amsmath}
\usepackage{latexsym}
\usepackage{mathrsfs}
\usepackage{amsthm}
\usepackage{verbatim}
\usepackage{amscd}
\usepackage{hyperref}
\usepackage{color}
\definecolor{shadecolor}{gray}{0.875}

\newtheorem{thrm}{Theorem}[section]
\newtheorem{lem}[thrm]{Lemma}
\newtheorem{cor}[thrm]{Corollary}
\newtheorem{prop}[thrm]{Proposition}

\newtheorem*{thrma}{Theorem A}

\newtheorem*{thrmb}{Theorem B}

\theoremstyle{definition}
\newtheorem{defn}[thrm]{Definition}

\newtheorem{exmple}[thrm]{Example}
\newtheorem{rmk}[thrm]{Remark}

\newtheorem*{ack}{Acknowledgments}

\DeclareMathOperator{\Eff}{\overline{Eff}}

\DeclareMathOperator{\Supp}{Supp}

\DeclareMathOperator{\vol}{vol}

\DeclareMathOperator{\Vol}{Vol}

\DeclareMathOperator{\mult}{mult}

\DeclareMathOperator{\bdp}{{\bf B}_{+}^{\rm div}}

\input xy
\xyoption{all}

\newcommand{\factor}[2]{\left. \raise 1pt\hbox{\ensuremath{#1}} \right/
        \hskip -2pt\raise -3pt\hbox{\ensuremath{#2}}}

\numberwithin{equation}{thrm}

\begin{document}

\title{Volume  and Hilbert function of $\mathbb R$-divisors}

\author{Mihai Fulger}
\address{Department of Mathematics, Princeton University\\
Princeton, NJ \, \, 08544}
\address{Institute of Mathematics of the Romanian Academy, P. O. Box 1-764, RO-014700,
Bucharest, Romania}
\email{afulger@princeton.edu}

\author{J\'anos Koll\'ar}
\address{Department of Mathematics, Princeton University\\
Princeton, NJ \, \, 08544}
\email{kollar@math.princeton.edu}
\author{Brian Lehmann}
\address{Department of Mathematics, Boston College  \\
Chestnut Hill, MA \, \, 02467}
\email{lehmannb@bc.edu}

\maketitle

\section{Introduction}

Let $X$ be a proper, normal algebraic variety of dimension $n$ over a 
field $K$ and $D$ an $\mathbb R$-divisor on $X$.
The {\it Hilbert function} of $D$ is the function
$$
{\mathcal H}(X,D): m\mapsto 
h^0( mD):=\dim_{K} H^0(X,\mathcal O_X(\lfloor mD\rfloor));
$$
defined for all $m\in \mathbb R$.
If $D$ is an ample Cartier divisor then ${\mathcal H}(X,D) $
agrees with the usual Hilbert polynomial whenever $m\gg 1$
is an integer, but in general
 ${\mathcal H}(X,D)$ is not  a polynomial, not  even if $D$ is a $\mathbb Z$-divisor and 
  $m\in \mathbb Z$.
The simplest numerical invariant associated to the  Hilbert function
is the {\it volume} of $D$, defined as 
 $$\vol(D):=\limsup_{m\to\infty}\frac{ h^0(mD)}{m^n/n!}.$$

If $E$ is an effective $\mathbb R$-divisor, then 
$$
h^0( mD-mE)\leq h^0( mD)\leq h^0( mD+mE)
\eqno{(*)}
$$
holds for every $m>0$, hence
$$
\vol(D-E)\leq  \vol(D)\leq \vol(D+E).\eqno{(**)}
$$
Furthermore, if equality holds in $(*)$ for every $m\gg 1$ then
equality holds in $(**)$. The aim of this note is to prove the converse
for {\it big} divisors, that is, when $\vol(D)>0$.  
Although the volume does not determine the Hilbert function,
we prove that
$$
\begin{array}{ll}
{\mathcal H}(X,D)\equiv {\mathcal H}(X,D-E)
\quad\Leftrightarrow\quad  \vol(D)=\vol(D-E)& \mbox{and}\\[0.5ex]
{\mathcal H}(X,D)\equiv {\mathcal H}(X,D+E)
\quad\Leftrightarrow\quad  \vol(D)=\vol(D+E).
\end{array}
$$
As a byproduct of the proof we also obtain a
characterization of such divisors $E$ in terms of the 
 negative part $N_{\sigma}(D)$ of  the {\it Zariski--Nakayama-decomposition}  (also called $\sigma$-decomposition) and of the divisorial part of the 
{\it augmented base locus} $\bdp(D)$; see
\cite{nakayama04}, \eqref{defnegativepart} and \eqref{aug.bs.defn} for definitions.

Another interesting consequence is that the answer depends
only on the $\mathbb R$-linear equivalence class of $D$.
This is obvious for $\mathbb Z$-linear equivalence, but
 it can easily happen that   $D'\sim_{\mathbb R}D$
yet  $h^0(X, mD)\neq h^0(X, mD')$ for every $m> 0$; see (\ref{Hf.lineq.exmp}). 
In fact, 
the only relationship between 
${\mathcal H}(X,D) $ and ${\mathcal H}(X,D') $ that we know of is 
$\vol(D)=\vol(D')$. 

Our main results are the following.

\begin{thrma} Let  $X$ be a proper, normal algebraic variety  over a perfect field, $D$ a big $\mathbb R$-divisor on $X$ and $E$ an effective $\mathbb R$-divisor on $X$. Then the following are equivalent.
\begin{enumerate}[i)]
\item $\vol(D-E)=\vol(D)$.
\item $E\leq N_{\sigma}(D)$.
\item $h^0(mD'-mE)=h^0(mD')$ for every $D' \sim_{\mathbb{R}} D$ and all $m>0$.
\item $h^0(mD-mE)=h^0(mD)$ for all $m>0$.
\end{enumerate}
Furthermore, if  $D$ is  $\mathbb R$-Cartier and nef then these are also equivalent to 
\begin{enumerate}[  v)]
\item $E=0$.
\end{enumerate}
\end{thrma}

\begin{thrmb} Let  $X$ be a proper, normal algebraic variety over a perfect field, $D$  a big $\mathbb R$-divisor on $X$ and $E$ an effective $\mathbb R$-divisor on $X$. Then the following are equivalent.
\begin{enumerate}[i)]
\item $\vol(D+E)=\vol(D)$.
\item $\Supp(E)\subseteq\bdp(D)$.
\item $h^0(mD'+rE)=h^0(mD')$ for every $D' \sim_{\mathbb{R}} D$ and $m, r>0$.
\item $h^0(mD+mE)=h^0(mD)$ for all $m>0$.
\end{enumerate}
Furthermore, if  $D$ is  $\mathbb R$-Cartier and nef then these are also equivalent to 
\begin{enumerate}[ v)]
\item $D^{n-1}\cdot E=0$.
\end{enumerate}
\end{thrmb}

Special cases of these theorems were first conjectured in connection with
the numerical stability criteria for familes of canonical models
of varieties of general type  \cite{k-modbook}. In trying to prove these,
we gradually realized that the above results hold and the
general setting led to  shorter proofs.

The theorems are proved in Section \ref{mainresultsec} but the
necessary technical background results involving $\mathbb R$-divisors, 
 the Zariski--Nakayama-decomposition 
and of the  augmented base locus
on singular varieties are left to 
Sections \ref{weildivsec} through \ref{s:bdp}.
Much of the relevant literature works with smooth projective
varieties over $\mathbb C$ but many of these proofs apply
in more general settings. We went through them and we state clearly which parts work for normal varieties in any characteristic. We also establish several results that show how to reduce similar types of questions
to smooth and projective varieties. These should be useful in similar contexts.

\begin{ack} We thank R.~Lazarsfeld and S.~Pal for
comments, discussions and references. 
Partial financial support  to JK  was provided  by  the NSF under grant number
 DMS-1362960.
\end{ack}

\section{Proofs of the Theorems} \label{mainresultsec}

\begin{prop}\label{prop:Zunique2}Let $X$ be a normal proper variety over an algebraically closed field and $D$ a big $\mathbb R$-divisor.  Suppose that $D=P+N$ with $\vol(P)=\vol(D)$ and $N$ effective.  Then $N \leq N_{\sigma}(D)$.
\end{prop}

The proof is a modification of \cite[Prop.5.3]{fl13z}.

\begin{proof}
By Corollary \ref{cor:makingcartier}, we may find a projective birational model $X'$ and $\mathbb R$-Cartier $\mathbb R$-divisors $D'$ and $P'$ on $X'$ such that for any positive real $m$ the pushforward of $\mathcal{O}_{X'}(mD')$ and $\mathcal{O}_{X'}(mP')$ are respectively $\mathcal{O}_{X}(mD)$ and $\mathcal{O}_{X}(mP)$, and the difference $D'-P'$ is effective.  Note that $D'$ and $P'$ still satisfy the hypotheses of the theorem.  If we prove the statement on $X'$, we can conclude the statement on $X$ by pushing forward and applying  Lemma \ref{lem:stricttransformgamma}.  So without loss of generality we may assume that $P$ and $D$, and hence $N$, are $\mathbb R$-Cartier $\mathbb R$-divisors and that $X$ is projective.


If $\pi:Y\to X$ is a generically finite proper morphism from a normal projective variety $Y$, then 
$$\pi_*N_{\sigma}(\pi^*D) = (\deg\pi)\cdot N_{\sigma}(D)$$
by Lemma \ref{lem:nspush}.$ii)$. Furthermore $$\vol(\pi^*D)=(\deg\pi)\cdot\vol(D)$$ by Theorem \ref{volprop}.$ii)$, the homogeneity of $\vol$, and \cite[Prop.2.9.(1)]{kur06} (the proof there does not use the assumption that the characteristic is zero). Therefore after passing to a nonsingular alteration (cf. \cite{dejong96}), it is enough to consider the case when $X$ is nonsingular and projective.

By assumption the volume of $P$ does not change if we add a small multiple of $N$.  Thus by \cite[Theorem 5.6]{cutkosky13} (see also \cite[Thm.A]{bfj09} and \cite[Cor.C]{lm09}), 
$$\langle P^{n-1} \rangle \cdot N=0,$$
where $\langle P^{n-1}\rangle$ is the positive intersection product defined in \cite{cutkosky13}, inspired by \cite{bfj09} and classical work of Matsusaka (\cite[p.1031]{Matsusaka72}; see also  \cite[p.515]{matsusaka75})

As in the proof of \cite[Thm.4.9]{bfj09}, it follows that for any ample $\mathbb R$-Cartier $\mathbb R$-divisor $A$ on $X$ and any small $\epsilon > 0$, we have 
$$\Supp(N) \subseteq \Supp\bigl(N_{\sigma}(P-\epsilon A)\bigr).$$
(Otherwise from $P=\frac{\epsilon}2A+(\frac{\epsilon}2A+P_{\sigma}(P-\epsilon A))+N_{\sigma}(P-\epsilon A)$ we get $P\geq_{N_i}\frac{\epsilon}2A$ for any component $N_i$ of $N$, i.e. $P-\frac{\epsilon}2A$ is numerically equivalent to an effective $\mathbb R$-divisor that does not contain $N_i$ in its support. Using \cite[Rem.4.5]{bfj09}, we see that $\frac{\epsilon^{n-1}}{2^{n-1}}A^{n-1}\cdot N\leq \langle P^{n-1}\rangle|_N\leq\langle P^{n-1}\rangle\cdot N$, but the LHS is only zero when $N=0$.)

In particular, Lemma \ref{lem:negsections2} shows that $N_{\sigma}(P - \epsilon A + N) = N_{\sigma}(P - \epsilon A) + N$.  Letting $\epsilon$ tend to $0$ and using continuity of $\sigma$ as in Lemma \ref{lem:sigmaproperties}.$iv)$, we see that $N_{\sigma}(D) = N_{\sigma}(P) + N$.
\end{proof}

We reduce our main theorems to the case
where the base field is algebraically closed.

\begin{rmk}\label{rmk:closedperfect}
Let  $K$ be a field 
and $L/K$ a separable field extension. Base change to $L$ is denoted by
the subscript $L$. 
If $X_K$ is a proper, normal algebraic variety over $K$ then
$X_L$ is a disjoint union of  proper, normal algebraic varieties over $L$.
If $E_K\subset X_K$ is a prime divisor then $E_L\subset X_L$ is a 
sum of prime divisors, each appearing with coefficient 1. 
Thus if $D_K$ is an $\mathbb R$-divisor on $X_K$ then
$\lfloor D_K\rfloor_L=\lfloor D_L\rfloor$. Thus
$$
\bigl(\mathcal O_{X_K}(D_K)\bigr)_L=\mathcal O_{X_L}(D_L)
\quad\mbox{and}\quad
h^0(D_K)=h^0(D_L).
\eqno{(\ref{rmk:closedperfect}.1)}
$$
Similarly, if $D_K$ is a $\mathbb Z$-divisor then
$|D_K|_L=|D_L|$ and hence the base locus commutes with separable field 
extensions. Using the characterization given in Lemma \ref{lem:sigmaproperties}.i)
and  Lemma \ref{lem:bdpzariski} 
 this implies
that 
$$
N_{\sigma}(D_L)=\bigl(N_{\sigma}(D_K)\bigr)_L\quad\mbox{and}\quad
\bdp(D_L)=\bigl(\bdp(D_K)\bigr)_L.
\eqno{(\ref{rmk:closedperfect}.2)}
$$
(If $X_K$ is geometrically normal but $L/K$ is not separable then
it can happen that $\lfloor D_K\rfloor_L\neq \lfloor D_L\rfloor$.
However (\ref{rmk:closedperfect}.2) still holds.)

If Theorems A and B hold for proper, normal varieties
over an algebraically closed field then they clearly also hold
for proper, normal, equidimensional schemes
over an algebraically closed field. Thus, by the above considerations,
 they hold for proper, normal varieties
over any perfect field.
\end{rmk}

\subsection*{Proof of Theorem A}
By Remark \ref{rmk:closedperfect} we may work over an algebraically closed
field.
The implications $ii)\to iii)\to iv) \to i)$ are immediate, while $i) \to ii)$ is Proposition \ref{prop:Zunique2}. Any nef $\mathbb R$-Cartier $\mathbb R$-divisor $D$ is movable,
i.e.~$N_{\sigma}(D)=0$. Then the equivalence between $ii)$ and $v)$ is clear.
\qed


\begin{rmk}The work of \cite{KL15} hints to an approach to Theorem A
using the theory of Okounkov bodies.
\end{rmk}

\begin{rmk} Related cases of Theorem A include:
\begin{enumerate}[$i)$]
\item If $D$ is an $\mathbb R$-Cartier $\mathbb R$-divisor, then in $iii)$ we may set $D'$ to be any $\mathbb R$-Cartier $\mathbb{R}$-divisor numerically equivalent to $D$.
\item If $X$ is nonsingular and projective over an algebraically closed field, if $D$ is big and \emph{movable}, and $E$ is \emph{pseudoeffective} (i.e. its numerical class is in the closure of the effective cone), then $\vol(D-E)=\vol(D)$ if and only if $E=0$.
\end{enumerate}
The first statement is a consequence of Lemma \ref{lem:sigmaproperties}.$iv)$. 
For the second, by \cite[Prop.5.3]{fl13z} we get 
$$P_{\sigma}(D-E)+(N_{\sigma}(D-E)+E)\equiv D=P_{\sigma}(D)\equiv P_{\sigma}(D-E).$$
Consequently $N_{\sigma}(D-E)+E\equiv 0$. Since the pseudoeffective cone is pointed (e.g.\ by \cite[Lem.2.4]{chms13}), it follows that $E=0$.\qed
\end{rmk}

\subsection*{Proof of Theorem B}
As in Theorem A, we may work over an algebraically closed field.
The implications $iii)\to iv)\to i)$ are immediate. Part $ii)$ of Theorem A and Lemma \ref{lem:negsections2} 
prove $i)\to iii)$.

Assume ${\rm Supp}(E)\subseteq\bdp(D)$. 
Let $A$ be ample in codimension 1 (cf. Definition \ref{def:amplecod1}). 
By Lemma \ref{lem:bdpzariski} and Lemma \ref{lem:bdpalternate}, 
we have 
${\rm Supp}(E)\subseteq{\rm Supp}(N_{\sigma}(D-\epsilon A))$ 
for arbitrarily small $\epsilon>0$. 
By Lemma \ref{lem:negsections2}, we see that 
$\vol(D+E-\epsilon A)=\vol(D-\epsilon A)$ for sufficiently small $\epsilon > 0$.
If $D$, $E$ and $A$ are $\mathbb{R}$-Cartier, we can conclude $\vol(D+E) = \vol(D)$ by the continuity of volumes for $\mathbb R$-Cartier $\mathbb R$-divisors.  To show that $\vol(D+E)=\vol(D)$ in general, we reduce to the $\mathbb{R}$-Cartier case by applying Theorem \ref{volprop}.$ii)$ and Corollary \ref{cor:makingcartier}.  Hence $ii)\to i)$.

Let $F$ be an irreducible component of $E$ and assume 
$F\not\subset{\rm Supp}(N_{\sigma}(D-\epsilon A))$. Then by Lemma \ref{concretesigma} there exists $m>0$ such that
$$mD+F=\bigl(\tfrac12{m\epsilon}A+F\bigr)+\bigl(\tfrac12{m\epsilon}A+
mP_{\sigma}(D-\epsilon A)\bigr)+mN_{\sigma}(D-\epsilon A)
$$
is $\mathbb R$-linearly equivalent to an effective divisor 
that does not contain $F$ in its support.
In particular $h^0(mD'+rE)\geq h^0(mD'+F)>h^0(mD')$ for some $D'\sim_{\mathbb R}D$ and some $r>0$, e.g. $r=\frac 1{{\rm mult}_F(E)}$.
Therefore $iii)\to ii)$.

Suppose now that $D$ is a big and nef $\mathbb R$-Cartier $\mathbb R$-divisor. 
Let $\pi:Y\to X$ be a proper birational morphism with $Y$ projective. By Lemma \ref{lem:wlift} there exists an effective 
$\pi$-exceptional divisor $F$ on $Y$ such that $\vol(D+E)=\vol(\pi^*D+\overline E+F)$, where $\overline E$ is a divisor with $\pi_*\overline E=E$. We can make choices such that $\overline E$ and $F$ are $\mathbb R$-Cartier $\mathbb R$-divisors. Of course 
$\vol(D)=\vol(\pi^*D)$.

If $\vol(D+E)=\vol(D)$, then $\vol(\pi^*D+\overline E +F)=\vol(\pi^*D)$. By \cite[Theorem 5.6]{cutkosky13}, we get $\langle \pi^*D^{n-1}\rangle\cdot(\overline E+F)=0$. Since $D$ is nef, we have $(\pi^*D)^{n-1}=\langle (\pi^*D)^{n-1}\rangle$ from \cite[Proposition 4.11]{cutkosky13}. By the projection formula $D^{n-1}\cdot E=0$. 

Conversely, if $D^{n-1}\cdot E=\pi^*D^{n-1}\cdot (\overline E+F)=0$, then \cite{luo90} shows that $h^0(\pi^*D+\overline E+F)=h^0(\pi^*D)$ (the analogous equality also holds for multiples). The proof there is carried out with $\mathbb Z$-coefficients and over base fields of characteristic zero, but extends to $\mathbb R$-coefficients over arbitrary algebraically closed base fields. We conclude by pushing forward to $X$.
\qed

\begin{rmk}As in the previous theorem, if $D$ is an $\mathbb R$-Cartier $\mathbb R$-divisor, then in $iii)$ we may set $D'$ to be any $\mathbb R$-Cartier $\mathbb R$-divisor numerically equvialent to $D$.  In fact even in the $\mathbb R$-Weil case we may replace $D'\sim_{\mathbb R}D$ with $D'-D$ being a numerically trivial $\mathbb R$-\emph{Cartier} $\mathbb R$-divisor (cf. Lemma \ref{lem:sigmaproperties}.$iv)$).
\end{rmk}

As mentioned in the introduction, if $D'\sim_{\mathbb R}D$, there is no
clear connection between the Hilbert functions $\mathcal H(X,D)$ and
$\mathcal H(X,D')$ other than that $\vol(D)=\vol(D')$ (cf. Theorem \ref{volprop}.$iv)$):

\begin{exmple}\label{Hf.lineq.exmp}
Let $S\to \mathbb P^1$ be a minimal ruled surface with a negative section
$E\subset S$ and a positive section  $C\subset S$ that is disjoint from
$E$. Let $F_1,\dots, F_4$ be distinct fibers.
Then $$C\sim_{\mathbb R} C+ (F_1-F_2)+\sqrt{2}(F_3-F_4).$$
Note that
$\lfloor mC+ m(F_1-F_2)+m\sqrt{2}(F_3-F_4)\rfloor  $
has negative intersection with $E$ for all real $m>0$. This implies that
$$
h^0\bigl(S,\mathcal O_S( mC+ m(F_1-F_2)+m\sqrt{2}(F_3-F_4))\bigr)<
h^0\bigl(S,\mathcal O_S( mC)\bigr)
$$
for every $m>0$.\qed
\end{exmple}

\section{Weil divisors} \label{weildivsec}
Let $X$ be a normal variety over a field. 
The basics of the theory of Weil $\mathbb R$-divisors  can be found
in \cite{schwede}.
An  {\it $\mathbb R$-divisor} 
(also called Weil $\mathbb R$-divisor or
$\mathbb R$-Weil $\mathbb R$-divisor)
is an $\mathbb R$-linear combination of  prime divisors. $D$  is \emph{effective}, denoted $D\geq 0$, if it is a nonnegative combination of prime divisors on $X$. If $D\geq E$, i.e. $D-E\geq 0$, we say that $D$ \emph{dominates} $E$. For an  $\mathbb R$-divisor $D$, the rule 
$$
U\mapsto H^0(U,D):=\bigl\{f\in K(X)^*\ |\ ({\rm div}(f)+D)|_U\geq 0\bigr\}\cup\bigl\{0\bigr\}
$$
defines a coherent sheaf $\mathcal O_X(D)$ on $X$. This coincides with the classical notation when $D$ is a  $\mathbb Z$-divisor. Note that $\mathcal O_X(D)=\mathcal O_X(\lfloor D\rfloor)$. If $D\geq 0$, then $\mathcal O_X(-D)$ is an ideal sheaf in $\mathcal O_X$. If $M$ is a Cartier $\mathbb Z$-divisor, then $\mathcal O_X(D+M)\simeq\mathcal O_X(D)\cdot\mathcal O_X(M)\simeq\mathcal O_X(D)\otimes\mathcal O_X(M)$ for any $\mathbb R$-divisor $D$.  

If $D$ and $D'$ are $\mathbb R$-divisors such that $D'-D={\rm div}(f)$ for some $f\in K(X)$, we say that $D$ and $D'$ are \emph{linearly equivalent} and denote this relation by  $D\sim D'$ or $D\sim_{\mathbb Z}D'$. Denote by $|D|$ the complete linear series $\{D'\ |\ D'\geq 0,\ D'\sim_{\mathbb Z}D\}$. It coincides with $|\lfloor D\rfloor|+\{D\}$, where $\{D\}$ denotes the fractional part of $D$. If $mD\sim mD'$ for some $m\in\mathbb Z^*$, we write $D\sim_{\mathbb Q} D'$. If $D'-D=\sum_{i=1}^ra_i{\rm div}(f_i)$ for some $r\in\mathbb N^*$, some $a_i\in\mathbb R$ and $f_i\in K(X)$, we write $D\sim_{\mathbb R}D'$.
Denote by $|D|_{\mathbb Q}$ and $|D|_{\mathbb R}$ the set of effective $\mathbb R$-divisors $D'$ that are $\mathbb Q$-linearly and respectively $\mathbb R$-linearly equivalent to $D$. If $D\sim D'$, then $H^0(X,D)\simeq H^0(X,D')$ and if $D\sim_{\mathbb Q}D'$, then $H^0(X,mD)\simeq H^0(X,mD')$ for sufficiently divisible $m$. 
However, no obvious connection seems to exists between $H^0(X,D)$ and $H^0(X,D')$ if $D\sim_{\mathbb R}D'$.

An $\mathbb R$-divisor $H$ is {\it ample} if 
$H=\sum_ia_i(H_i+{\rm div}(f_i))$, where $a_i\in\mathbb R_+$, where $f_i\in K(X)$, and $H_i$ are effective ample Cartier $\mathbb Z$-divisors.  Note that an ample $\mathbb{R}$-divisor is always $\mathbb{R}$-Cartier, and that this definition coincides with the classical one in \cite[\S 2]{lazarsfeld04}.

Two $\mathbb{R}$-Cartier $\mathbb{R}$-divisors are numerically equivalent if they have the same intersection against every proper curve in $X$.  

We review some of the basic theory of  $\mathbb R$-divisors.  Over $\mathbb C$, many of the results in this section appear in \cite[\S II]{nakayama04} or \cite{fujino09}.

\begin{lem}\label{lem:effCartier}Let $X$ be a normal variety and 
$D$ an effective $\mathbb R$-Cartier $\mathbb R$-divisor. 
Then $D$ is a positive $\mathbb R$-linear combination $\sum_ia_iD_i$ of effective Cartier divisors.
\end{lem}

\begin{proof}The argument in \cite[Lem.0.14]{fujino09} is characteristic free.
\end{proof}

\begin{lem}\label{lem:pushpullCartier}Let $\pi:Y\to X$ be a proper birational morphism of normal varieties, 
and  $D$  an $\mathbb R$-Cartier $\mathbb R$-divisor on $X$. 
Then $\pi_*\mathcal O_Y(\pi^*D+E)=\mathcal O_X(D)$ 
for any effective $\pi$-exceptional $\mathbb R$-divisor $E$.
\end{lem}

\begin{proof}The argument is similar to \cite[Lem.2.11]{nakayama04}. 
Let $U\subset X$ be open and $f\in K(X)^*$. 
By the projection formula \cite[Prop.2.3.(c)]{fulton84}, if 
${\rm div}_Y(f)+\pi^*D+E\geq 0$ over $\pi^{-1}U$, then 
${\rm div}_X(f)+D\geq 0$ over $U$. 
By Lemma \ref{lem:effCartier}, we see that if 
${\rm div}_X(f)+D\geq 0$ on $U$, then 
${\rm div}_Y(f)+\pi^*D\geq 0$ on $\pi^{-1}U$. 
In particular ${\rm div}_Y(f)+\pi^*D+E\geq 0$ on $\pi^{-1}U$.
\end{proof}


The following lemma can be used to reduce many questions
involving the shaves $\mathcal O_X(D)$ to 
normal {\it projective} varieties.

\begin{lem}\label{lem:wlift}Let $\pi:Y\to X$ be a proper birational morphism of normal varieties
and  $D_i$ a finite collection of  $\mathbb R$-divisors on $X$. 
Then there are $\mathbb R$-divisors $D_i^Y$ on $Y$ such that
 $\pi_*D_i^Y=D_i$ for every $i$ and
$$
\pi_*\mathcal O_Y\bigl(F+\pi^*M+\textstyle{\sum}_i m_i D_i^Y\bigr)=
\mathcal O_X\bigl(M+\textstyle{\sum}_i m_i D_i\bigr)
$$
for every  $m_i\in \mathbb R^+$, effective $\pi$-exceptional $\mathbb R$-divisor $F$ on $Y$, and $\mathbb R$-Cartier $\mathbb R$-divisor $M$ on $X$.
\end{lem}

\begin{proof}If the statement is true for $F=0$, then it is true for every $F \geq 0$, so we assume that $F=0$ throughout. The question is local on $X$, so we may also assume that $X$ is affine. Let $E$ be the reduced
Weil divisor whose support is the divisorial component of the exceptional locus of $\pi$. For $D$ an $\mathbb R$-divisor on $X$, and $\overline D$ an $\mathbb R$-divisor on $Y$ with $\pi_*\overline D=D$, we have $\mathcal O_X(D)=\bigcup_{r\geq 0}\pi_*\mathcal O_Y(\overline{D}+rE)$. Then by coherence there exists $r_D$ such that 
\begin{equation}\label{eq:h0stop}
\pi_*\mathcal O_Y\bigl(\overline D+rE\bigr)=\mathcal O_X\bigl(D\bigr)\mbox{ for all }r\geq r_D.\end{equation}
Let $\phi$ be a regular function on $X$ such that 
$L:={\rm div}_X(\phi)\geq D_i$ for all $i$. 
Let $\overline D_i$ be $\mathbb R$-divisors on $Y$ such that 
$\pi_*\overline D_i=D_i$. 
For any $r \geq 0$, we have $\overline{D}_{i} + rE \leq E_{i}' + L$ for some effective $\pi$-exceptional $\mathbb{R}$-divisor $E_{i}'$.  By Lemma \ref{lem:pushpullCartier}, any global section of 
$\mathcal O_Y(\sum_im_i\overline D_i+rE)$ for any $r \geq 0$ is also a global section of 
$\mathcal O_X((\sum_im_i)L)$.
Thus the poles along $E$ of rational functions that are sections of 
$\sum_im_i\overline{D}_i+rE$ are bounded below by $-(\sum_im_i)\pi^*L$. 
This implies that there exists $r>0$ such that $H^0(Y,\sum_im_i(\overline{D}_i+(r+t)E))$ 
is independent of $t\geq 0$ for each $m_i\geq 0$. 
In particular it is equal to $H^0(\mathcal O_X(\sum_im_iD_i))$ by \eqref{eq:h0stop}.
Since $X$ is affine, this implies 
$\pi_*\mathcal O_Y(\sum_im_i(\overline{D}_i+rE))=\mathcal O_X(\sum_im_iD_i)$. Set $D^Y_i:=D_i+rE$.

We now show that if $M$ is an $\mathbb R$-Cartier $\mathbb R$-divisor on $X$, then $\pi^*M+\sum_im_iD^Y_i\geq 0$ if and only if $M+\sum_im_iD_i\geq 0$. Up to replacing 
$M$ by $M+{\rm div}_X(f)$, this completes the proof.
One implication is clear by the projection formula. 
Assume now $M+\sum_im_iD_i\geq 0$. 
If $M$ is a $\mathbb Q$-Cartier $\mathbb Q$-divisor, 
then $uM$ is a Cartier divisor for some positive integer $u$, 
and by the projection formula, 
$$\pi_*\mathcal O_Y\bigl(\pi^*(uM)+\textstyle{\sum}_i(um_i)D_i^Y\bigr)=
\mathcal O_X\bigl(uM\bigr)\otimes\pi_*\mathcal O_Y\bigl(\textstyle{\sum}_i(um_i)D_i^Y\bigr)=
\mathcal O_X\bigl(u(M+\textstyle{\sum}_im_iD_i)\bigr)$$ 
for all $m_i\in\mathbb R^+$. 
Thus if $1$ is a section of $\mathcal O_X(u(M+\sum_im_iD_i))$, 
then it is also a section of $\mathcal O_Y(u(\pi^*M+\sum_im_iD_i^Y))$. 

Assume now $M=\sum_ja_jM_j$ is an $\mathbb R$-combination of Cartier divisors, 
with $M+\sum_im_iD_i\geq 0$. 
We may further assume that $D_i$ or $-D_i$ is a prime divisor for each $i$.
As a condition on the $m_i$'s and $a_j$'s, the effectivity of $M+\sum_im_iD_i$ is a
system of linear inequalities with integer coefficients. Any of its real solutions can be
approximated arbitrarily close by rational solutions. 
We conclude from the case when $M$ is a $\mathbb Q$-Cartier $\mathbb Q$-divisor
by taking limits coefficientwise.
\end{proof}

The following corollary allows us to reduce questions about $\mathbb R$-divisors to $\mathbb R$-Cartier $\mathbb R$-divisors.

\begin{cor} \label{cor:makingcartier}
Let $D_i$ be a finite set of $\mathbb R$-divisors on a normal variety $X$. 
Then there exist a quasiprojective, normal variety $Y$, a proper birational morphism $\pi:Y\to X$ and $\mathbb R$-Cartier $\mathbb R$-divisors $D^Y_i$ on $Y$ such that $\pi_*D^Y_i=D_i$ and
$$
\pi_*\mathcal O_Y\bigl(G+\pi^*M+\textstyle{\sum}_im_iD^Y_i\bigr)=
\mathcal O_X\bigl(M+\textstyle{\sum}_im_iD_i\bigr)$$ 
holds for all $m_i\in\mathbb R^+$, all effective $\pi$-exceptional $\mathbb R$-divisors $G$ on $Y$ and all $\mathbb R$-Cartier $\mathbb R$-divisors $M$ on $X$.
\end{cor}

\begin{proof}
We may assume that $D_i$ or $-D_i$ is a prime divisor for each $i$. 
Successively normalize the blow-up of the birational transform 
 of each $D_i$, obtaining a birational morphism 
$f:Z\to X$ with $\mathbb R$-Cartier $\mathbb R$-divisors $D'_i$ such that $f_*D'_i=D_i$. 
Let $g:Y\to Z$ be the normalized blow-up of the exceptional locus of $f$. 
Let $\pi=f\circ g$. Then $\overline D_i:=g^*D'_i$ is an 
$\mathbb R$-Cartier $\mathbb R$-divisor with $\pi_*\overline D_i=D_i$. 
The relative $\mathcal O(-1)$ for $g$ is an effective Cartier divisor $F$ 
whose support is the exceptional locus of $\pi$. 
As in the proof of Lemma \ref{lem:wlift}, for $r\gg0$, 
we may set $D_i^Y:=\overline D_i+rF$.
To obtain $Y$ quasiprojective, apply Chow's Lemma and normalize.
\end{proof}

We have defined $\vol(D):=\limsup_{m\to\infty}\frac{h^0(mD)}{m^n/n!}$.
For $\mathbb R$-Cartier $\mathbb R$-classes on projective varieties, this definition of volume differs from the classical one (cf.\ \cite[Cor.2.2.45]{lazarsfeld04}). The definitions coincide for $\mathbb Z$-classes but 
in \cite{lazarsfeld04}
the volume of $\mathbb Q$-classes is defined by homogeneous extension from $\mathbb Z$ and for $\mathbb R$-classes it is given by continuous extension from $\mathbb Q$. 
We check that the definitions in fact agree. We also check that we can replace $\limsup$ by $\lim$.

\begin{thrm} \label{volprop}
Let $D$ be an  $\mathbb R$-divisor on a proper normal variety $X$ of dimension $n$. Then
\begin{enumerate}[i)]
\item $\vol(D)=\lim_{m\to\infty}\frac{h^0(mD)}{m^n/n!}$.
\item If $D$ is an $\mathbb R$-Cartier $\mathbb R$-divisor, then $\vol(D)$ agrees with the definition in \cite[Cor.2.2.45]{lazarsfeld04}.
\item (Kodaira lemma)  $\vol(D)>0$ if and only if for every  $\mathbb R$-divisor $B$ there exists $\epsilon>0$ and an effective  $\mathbb R$-divisor $C$ such that $D\sim_{\mathbb Q}\epsilon\cdot B+C$.
\item If $D'$ is an $\mathbb R$-divisor on $X$ such that $D'-D$ is a numerically trivial $\mathbb R$-Cartier $\mathbb R$-divisor, then $\vol(D)=\vol(D')$.
\end{enumerate}
\end{thrm}

Most of the references used in the proof work over $\mathbb C$. \cite[\S2.2]{cutkosky13} and the references therein explain how to extend these
to arbitrary fields.

\begin{proof}
By Corollary \ref{cor:makingcartier}, we may assume that $X$ is projective and $D$ (hence also $D'$) and $B$ are $\mathbb R$-Cartier $\mathbb R$-divisors. Then there exists an ample $\mathbb Z$-divisor $H$ with $D\leq H$. Hence $H^0(X,mD)$ is a graded linear series. If $D$ is not big, then the limit is zero. Otherwise by \cite[Thm.1.2]{cutkosky13b} we have 
\begin{equation}\label{eq:vollim}\vol(D)=\lim_{m\to\infty}\frac{h^0(m\cdot m_0D)}{(m\cdot m_0)^n/n!}<\infty,\end{equation}
where $m_0={\rm gcd}\{m\in\mathbb Z\ |\ h^0(mD)\neq 0\}$.
We will return to showing that $m_0=1$. 

For now we prove $ii)$ and $iii)$. Provisionally denote by $\Vol(D)$ the volume of the $\mathbb R$-Cartier
$\mathbb R$-divisor $D$ in the sense of \cite[Cor.2.2.45]{lazarsfeld04}.
From \eqref{eq:vollim} we see that $\vol$ is also homogeneous, so that for a $\mathbb{Q}$-Cartier $\mathbb Q$-divisor $D$ we have $\vol(D) = \Vol(D)$.  

We first show that if $\vol(D)>0$, then $\vol(D)=\Vol(D)$.
By homogeneity we may assume $h^0(D)>0$. Then $D=E+{\rm div}(f)$ for some effective $\mathbb R$-Cartier $\mathbb R$-divisor $E$ and for some rational function $f$ on $X$.
By Lemma \ref{lem:effCartier}, we have
$E=\sum_ia_iE_i$ for some positive $a_i\in\mathbb R$ and effective Cartier $\mathbb Z$-divisors $E_i$. Then
$$\tfrac 1{m^n}\vol\bigl(\textstyle{\sum}_i\lfloor ma_i\rfloor E_i+{\rm div}(f^m)\bigr)\leq \vol(D)\leq \tfrac 1{m^n}\vol\bigl(\textstyle{\sum}_i\lceil ma_i\rceil E_i+{\rm div}(f^m)\bigr).$$
The LHS and RHS both converge to $\Vol(D)$ as $m$ grows. Furthermore if $\vol(D)>0$, then $\sum_i\lfloor ma_i\rfloor E_i+{\rm div}(f^m)$ is a big Cartier $\mathbb Z$-divisor for large enough $m$, hence it dominates some ample $\mathbb Q$-divisor by Kodaira's Lemma (cf. \cite[Cor.2.2.7]{lazarsfeld04}).

It remains to show that if $\Vol(D)>0$, then $\vol(D)>0$. First observe that if $\Vol(D)>0$, then $D$ is big in the sense of \cite[\S 2.2.B]{lazarsfeld04}, i.e. $D$ dominates an ample $\mathbb R$-divisor. Indeed by continuity (cf.\ \cite[Cor.2.2.45]{lazarsfeld04}) there exists a small ample $\mathbb R$-divisor $H$ such that $D-H$ is a $\mathbb Q$-Cartier $\mathbb Q$-divisor with $\Vol(D-H)>0$. Then the claim follows from Kodaira's Lemma. We can write 
$$D=\textstyle{\sum}_ia_i(H_i+{\rm div}(f_i))+\textstyle{\sum}_jb_jE_j,$$
where $H_i$ are ample effective $\mathbb Z$-divisors, $f_i$ are rational functions, $E_j$ are effective $\mathbb R$-Cartier $\mathbb Z$-divisors, $a_i$ and $b_j$ are positive real numbers.
Let $F$ be the union of the supports of ${\rm div}(f_i)$. There exists a real number $N>0$ such that
$\{ma_i\}{\rm div}(f_i)>-N\cdot F$ for all $i$ and all $m$. Furthermore there exists a positive integer $r$ such that for each $i$ the Weil divisor $rH_i-N\cdot F$ has a section given by some rational function $g_{i}$. In particular $$rH_i+\{ma_i\}{\rm div}(f_i)+{\rm div}(g_{i})>0.$$ Then
$$mD> \textstyle{\sum}_i\bigl((\lfloor ma_i\rfloor-r)H_i+\lfloor ma_i\rfloor{\rm div}(f_i)-{\rm div}(g_i)\bigr)+\textstyle{\sum}_j\lfloor mb_j\rfloor E_j.$$
The RHS is an effective big Cartier $\mathbb Z$-divisor for $m$ sufficiently large, therefore $\vol(D)>0$. The proof of $ii)$ is complete. Part $iii)$ follows easily from the projective and $\mathbb R$-Cartier case. We have showed that if $D$ is an $\mathbb R$-Cartier $\mathbb R$-divisor, and $\Vol(D)=\vol(D)>0$, then $mD$ is effective for $m$ large enough. This proves that $m_0=1$, and completes the proof of $i)$.
The volume function $\Vol$ is defined on the real N\' eron--Severi space $N^1(X)_{\mathbb R}$, and then part $iv)$ follows.
\end{proof}


\section{Divisorial Zariski decompositions} \label{divzarsec}

Let $X$ be a normal proper variety over a field $K$.  Let $D$ be a big $\mathbb R$-divisor. Following Nakayama (\cite{nakayama04}), for $\Gamma$ a prime divisor on $X$ we define 
$$\sigma_{\Gamma}(D)=\inf\bigl\{{\rm mult}_{\Gamma}D'\ |\ D' \sim_{\mathbb{R}} D,\ D'\geq0\bigr\},$$
where we write $D\sim_{\mathbb{R}}D'$ if there exist rational functions $f_i$ on $X$ and \emph{real} numbers $a_i$ such that $D-D'=\sum_i a_i\cdot{\rm div}(f_i)$. The basic properties of $\sigma_{\Gamma}(D)$ are studied by \cite{nakayama04} for smooth projective varieties in characteristic $0$, and by \cite{Mus11}, \cite{chms13} for smooth projective varieties in arbitrary characteristic.  We make the brief verifications necessary to extend these results to normal proper varieties as well. We start with the projective case.

\begin{lem} \label{lem:sigmaproperties}
Let $X$ be a normal projective variety and  $D$  a big $\mathbb R$-divisor.  Fix a prime divisor $\Gamma$.
\begin{enumerate}[i)]
\item We also have $\sigma_{\Gamma}(D) = \inf\bigl\{{\rm mult}_{\Gamma}D'\ |\ D' \sim_{\mathbb{Q}} D,\ D'\geq0\bigr\}$ and $$\sigma_{\Gamma}(D)=\lim_{m\to\infty}\tfrac1{m}\min\bigl\{{\rm mult}_{\Gamma}D''\ |\ D''\sim_{\mathbb Z} mD,\ D''\geq 0\bigr\}.$$
\item Let $A$ be an ample $\mathbb R$-Cartier $\mathbb R$-divisor.  Then $\lim_{\epsilon \searrow 0} \sigma_{\Gamma}(D + \epsilon A) = \sigma_{\Gamma}(D)$.
\item The $\mathbb R$-divisor $F:= D - \sigma_{\Gamma}(D)\Gamma$ has $\sigma_{\Gamma}(F) = 0$ and $\sigma_{\Gamma'}(F) = \sigma_{\Gamma'}(D)$ for any other prime divisor $\Gamma'$.  Furthermore the natural inclusion $H^{0}(X,mF) \hookrightarrow H^{0}(X,mD)$ is an equality for any positive real number $m$.
\item If 
$L$ is a numerically trivial $\mathbb R$-Cartier $\mathbb R$-divisor then $\sigma_{\Gamma}(D + L) = \sigma_{\Gamma}(D)$.  The induced function $\sigma_{\Gamma}: N^{1}(X) \to \mathbb{R}$ sending a numerical class $\alpha \in N^{1}(X)$ to $\sigma_{\Gamma}(D + \alpha)$ is continuous in a sufficiently small neighborhood of $0$.
\end{enumerate}
\end{lem}

\begin{proof}
The proofs are analogous to \cite[Lem.III.1.4]{nakayama04} and \cite[Lem.III.1.7]{nakayama04}.
\end{proof}

We can usually reduce questions involving $\sigma_{\Gamma}$ to the projective case by using Lemma \ref{lem:wlift} and the following:

\begin{lem} \label{lem:stricttransformgamma}
Let $\pi: Y \to X$ be a  birational morphism of normal, proper varieties.  Suppose that $D$ is a big $\mathbb R$-divisor on $X$. Assume that one of the following holds:
\begin{enumerate}[i)]
\item There exists  a big $\mathbb R$-divisor $L$ on $Y$ with $\pi_{*}L=D$ such that  for every $\mathbb R$-Cartier $\mathbb R$-divisor $M$ on $X$
the condition $D+M\geq 0$ holds iff $L+\pi^*M\geq 0$.
\item $X$ and $Y$ are projective and there exists a big $\mathbb R$-divisor  $L$ on $Y$
such that $\pi_*\mathcal O_Y(mL)=\mathcal O_X(mD)$ for all integers $m\geq 0$. 
\end{enumerate}
Then for any prime divisor $\Gamma$ on $X$ we have $\sigma_{\Gamma}(D) = \sigma_{\Gamma'}(L)$ where $\Gamma'$ is the birational transform of $\Gamma$ on $Y$.
\end{lem}


\begin{proof}By letting $M$ range through the $\mathbb R$-linearly trivial divisors on $X$ we immediately obtain $i)$. Part $ii)$ is a consequence of Lemma \ref{lem:sigmaproperties}.$i)$ and the fact that $\pi_*$ induces an equality of global sections for sheaves.   
\end{proof}

\begin{rmk}
Let $X$ be a normal proper variety.  Suppose that $D$ is a big $\mathbb R$-Weil $\mathbb{R}$-divisor on $X$.  Then there are at most finitely many prime divisors $\Gamma$ such that $\sigma_{\Gamma}(D) > 0$ (since $0\leq \sigma_{\Gamma}(D)\leq{\rm mult}_{\Gamma}(D')$ for any fixed effective $D'\sim_{\mathbb R}D$.)\qed
\end{rmk}






We can now define
\begin{equation}\label{defnegativepart}
N_{\sigma}(D) = \sum_{\Gamma\mbox{ prime divisor on }X} \sigma_{\Gamma}(D) \cdot
\Gamma \qquad \mbox{and}\qquad P_{\sigma}(D) = D - N_{\sigma}(\Gamma).
\end{equation}
We call the decomposition $D=P_{\sigma}(D)+N_{\sigma}(D)$ the \emph{divisorial Zariski decomposition} of $D$. 

\begin{defn}We say that a big $\mathbb R$-divisor $D$ is \emph{movable} if $N_{\sigma}(D)=0$ or equivalently $D=P_{\sigma}(D)$.
\end{defn}

\begin{rmk}\label{movlim}Let $D$ be a big, movable $\mathbb R$-divisor on a normal proper variety $X$. Let $D'\sim_{\mathbb R}D$ with $D'\geq 0$. Then $D'=P_{\sigma}(D')$ is the componentwise limit
of the divisors $D'_m:=D'-\frac 1m\min\{{\rm mult}_{\Gamma}D''\ |\ D''\sim_{\mathbb Z} mD',\ D''\geq 0\}$.
(This is Lemma \ref{lem:sigmaproperties}.$i)$ when $X$ is projective, and we can reduce to this case via Lemma \ref{lem:stricttransformgamma}.$i)$.)
Observe that $|mD'_m|$ is a linear series without fixed divisorial components for large $m$. In this sense, we understand movable $\mathbb R$-divisors as limits of divisors moving in linear series without fixed divisorial components.
\end{rmk}

\begin{rmk}If $D$ is a big and nef $\mathbb R$-Cartier $\mathbb R$-divisor, then $D$ is movable. 
\end{rmk}

\begin{defn}\label{def:amplecod1} Let $X$ be a normal variety. An $\mathbb R$-divisor $A$ is {\it ample in codimension 1}
if there exists a closed subset $Z\subset X$ of codimension at least $2$
such that $A|_{X\setminus Z}$ is an ample $\mathbb R$-Cartier $\mathbb R$-divisor.
\end{defn}

The following lemma shows that all normal varieties admit such divisors. 

\begin{lem}\label{lem:pushacd1} Let $\pi:Y\to X$ be a proper, generically finite, dominant 
morphism of normal varieties, and $A$ an $\mathbb R$-divisor on $Y$ that is ample in codimension 1.
Then  $\pi_*A$ is ample in codimension 1.
\end{lem}
\begin{proof} By removing a suitable subset of codimension $2$ from $X$
we may assume that $\pi$ is finite and $A$ is ample on $Y$.  Note $\pi_{*}A$ is $\mathbb{R}$-Cartier in codimension $1$, so that by shrinking $Y$ and $X$ further we may assume $\pi_{*}A$ is also $\mathbb{R}$-Cartier.
If $B$ is a $\mathbb Q$-divisor whose multiples separate finite subsets on 
$Y$, then multiples of  $\pi_*B$ also separate finite subsets on 
$Y$. Thus  $\pi_*A$ is ample.
\end{proof}




\begin{lem}\label{concretesigma}Let $X$ be a normal proper variety over a field, 
 $\Gamma$  a prime divisor, 
and   $A$  an $\mathbb R$-divisor that is ample in codimension 1. 
Then
\begin{enumerate}[i)]
\item If $E$ is an $\mathbb R$-divisor, then, for $m$ sufficiently large,
$E+mA\sim_{\mathbb R}B_m$, for some $B_m\geq 0$ with 
$\Gamma\not\subset{\rm Supp}(B_m)$.
\item If $P$ is a big $\mathbb R$-divisor with $\sigma_{\Gamma}(P)=0$, 
then $P+A\sim_{\mathbb R}C$, for some $C\geq 0$ with 
$\Gamma\not\subset{\rm Supp}(C)$.
\end{enumerate}
\end{lem}

\begin{proof}
For $i)$, by working over the smooth locus of $X$ we see that $E+mA$ is
ample in codimension 1 for $m$ sufficiently large and then the statement
is clear.

Let $m$ be as in part $i)$ for $E=\Gamma$. By the definition of $\sigma_{\Gamma}$, there exists an effective $P_m\sim_{\mathbb R}P$
such that ${\rm mult}_{\Gamma}(P_m)\leq\frac 1m$. By $i)$, we have that
$P_m+A$ is $\mathbb R$-linearly equivalent to an effective 
$\mathbb R$-divisor $C$ without $\Gamma$ in its support.
\end{proof}

\begin{lem}\label{lem:movprop}Let $X$ be a normal proper variety. Then
\begin{enumerate}[i)]
\item If $D$ is a big $\mathbb R$-divisor, then $P_{\sigma}(D)$ is big and movable.
If $D$ is effective, then so is $P_{\sigma}(D)$.
\item If $P$ and $D$ are big $\mathbb R$-divisors with $P$ movable and $P\leq D$, then $P\leq P_{\sigma}(D)$.
\item If $\pi:Y\to X$ is a proper generically finite dominant morphism of normal proper varieties 
and $P$ is a big movable $\mathbb R$-divisor on $Y$, then $\pi_*P$ is also big and movable.
\item Let $A$ be an  $\mathbb R$-Cartier $\mathbb R$-divisor that is
ample in codimension 1. Then for every
 $\mathbb R$-divisor $E$ there exists $\epsilon_E>0$ such that $A-\epsilon_EE$ is big and movable. 
\end{enumerate}
\end{lem}

\begin{proof}Part $i)$ is a consequence of Lemma \ref{lem:sigmaproperties}.$iii)$ in the projective case, and can be reduced to this case in general by Lemma \ref{lem:stricttransformgamma}.$i)$ and Lemma \ref{lem:wlift}. 

For part $ii)$, assume $D=P+N$ with $N$ effective. By Lemma \ref{lem:stricttransformgamma}.$i)$, and Lemma \ref{lem:wlift} we can assume that $X$ is projective. Let $A$ be an effective ample divisor.
For all prime divisors $\Gamma$ on $X$ we have $$\sigma_{\Gamma}(D+\epsilon A)\leq\sigma_{\Gamma}(P)+\sigma_{\Gamma}(N+\epsilon A)=\sigma_{\Gamma}(N+\epsilon A).$$ By summing over all $\Gamma$'s we obtain $N_{\sigma}(D+\epsilon A)\leq N_{\sigma}(N+\epsilon A)$, and hence $P_{\sigma}(D+\epsilon A)\geq P$. The continuity property in Lemma \ref{lem:sigmaproperties}.$ii)$ implies $P_{\sigma}(D)\geq P$.

In $iii)$, observe first that any divisor ample in codimension 1 is big. Furthermore, 
an $\mathbb R$-divisor is big if and only if it dominates some divisor ample in codimension 1.
From Lemma \ref{lem:pushacd1} it follows that if $P$ is big, then $\pi_*P$ is also big.


To settle the movability of $\pi_*P$, by Lemmas \ref{lem:sigmaproperties}.$i)$,  \ref{lem:stricttransformgamma}.$i)$ and Remark \ref{movlim}, it is enough to show that if $V$ is a linear series without fixed divisorial components on $Y$, then $\pi_*V$ spans a linear series without fixed divisorial components on $X$. 
By Remark  \ref{rmk:closedperfect} we may assume that the base field is infinite.
If $\Gamma$ is a prime divisor on $X$, let $\Gamma'_i$ with $1\leq i\leq r$ be the divisorial components of $\pi^{-1}\Gamma$. If $\pi_*V$ spans a linear series with a fixed component $\Gamma$, then $\mult_{\Gamma}Q>\epsilon$ for all $Q\in\pi_*V$ and for some $\epsilon>0$ by the finite dimensionality of $V$. Then 
$V$ is the union of the proper subspaces $V_i=\{R\in V\ |\ \mult_{\Gamma_i}R>\frac{\epsilon}{r\cdot\deg{\pi}}\}$. This is impossible over an infinite field.

Since global sections are determined outside any codimension 2 subset, it is enough to consider the projective case of $iv)$. By the lower convexity of $N_{\sigma}$, it is enough to treat the case when $A$ and $E$ are $\mathbb Z$-divisors with $A$ ample Cartier. Then $\mathcal O_X(mA-E)\simeq\mathcal O_X(-E)\otimes\mathcal O_X(A)^{\otimes m}$ is globally generated for large $m$. In particular the linear series $|mA-E|$ has no fixed components and $N_{\sigma}(A-\frac 1mE)=0$.
\end{proof}

\begin{rmk}\label{rmk:finitepullback}When $\pi:Y\to X$ is a finite morphism of normal proper varieties, for every $\mathbb R$-divisor $D$ on $X$ we can define $\pi^*D$ as the closure in $Y$ of $\pi_U^*D_U$, where $U\subset X$ is the smooth locus, and $\pi_U:Y\times_XU\to U$ is the induced finite morphism. Since ${\rm codim}(X\setminus U,X)\geq 2$, we see that $\pi^*$ respects linear equivalence (with $\mathbb Z$, $\mathbb Q$, and $\mathbb R$ coefficients).
\end{rmk} 

\begin{lem}\label{lem:nspush}Let $\pi:Y\to X$ be a generically finite morphism of normal proper varieties 
and  $D$  a big $\mathbb R$-divisor on $X$. Then
\begin{enumerate}[i)]
\item If $\pi$ is finite, then $N_{\sigma}(\pi^*D)=\pi^*N_{\sigma}(D)$.
\item If $\pi$ is only generically finite, but $D$ is an $\mathbb R$-Cartier $\mathbb R$-divisor, then
$\pi_*N_{\sigma}(\pi^*D)=(\deg\pi)\cdot N_{\sigma}(D)$.
\end{enumerate}
\end{lem}

\begin{proof}If $\pi$ is finite, then $N_{\sigma}(\pi^*D)\leq \pi^*N_{\sigma}(D)$ because $H^0(Y,\pi^*D)\supseteq \pi^*H^0(X,D)$. When $\pi$ is only generically finite and $D$ is $\mathbb R$-Cartier, the same argument and the projection formula (cf. \cite[Proposition 2.3.(c)]{fulton84}) prove $\pi_*N_{\sigma}(\pi^*D)\leq (\deg\pi)\cdot N_{\sigma}(D)$. On the other hand $D=\frac 1{\deg\pi}\pi_*P_{\sigma}(\pi^*D)+\frac1{\deg\pi}\pi_*N_{\sigma}(\pi^*D)$, and $\pi_*P_{\sigma}(\pi^*D)$ is big and movable by Lemma \ref{lem:movprop}.$iii)$. By Lemma \ref{lem:movprop}.$ii)$ it follows that $\frac 1{\deg\pi}\pi_*P_{\sigma}(\pi^*D)\leq P_{\sigma}(D)$ and hence $\frac1{\deg\pi}\pi_*N_{\sigma}(\pi^*D)\geq N_{\sigma}(D)$. Therefore in both $i)$ and $ii)$, 
$$\pi_*N_{\sigma}(\pi^*D)=(\deg\pi)\cdot N_{\sigma}(D).$$
When $\pi$ is finite, this forces equality in $N_{\sigma}(\pi^*D)\leq \pi^*N_{\sigma}(D)$.
\end{proof}

\begin{lem}\label{lem:negsections2}If $D$ is a big $\mathbb R$-divisor on the proper normal variety $X$, and $E\geq 0$ with 
${\rm Supp (E)}\subset{\rm Supp}(N_{\sigma}(D))$, then
\begin{equation*}
N_{\sigma}(D+E) = N_{\sigma}(D)+E
\end{equation*}
and
$$H^{0}(X,D)=H^0(X,D+E)=H^0(X,P_{\sigma}(D)+E)=H^{0}(X,P_{\sigma}(D)).$$
\end{lem}

\begin{proof}We argue just as in \cite[III.1.8 Lemma]{nakayama04} and \cite[III.1.9 Corollary]{nakayama04}. When $X$ is not projective,
we replace the ample $A$ from the proof of \cite[III.1.8 Lemma]{nakayama04}
by a divisor ample in codimension 1.
\end{proof}



\section{Divisorial augmented base locus}
\label{s:bdp}

The augmented base locus of an $\mathbb R$-Cartier $\mathbb R$-divisor
on a normal complex projective variety $X$ is defined in 
\cite[Definition 1.2]{elmnp06} as 
${\bf B}_{+}(D)=\bigcap_{D=A+E}{\rm Supp}(E)$, 
where $A$ is an ample $\mathbb R$-divisor and $E$ is an effective 
$\mathbb R$-Cartier $\mathbb R$-divisor. 
For normal proper varieties, we mimic this construction
by using divisors ample in codimension 1.  The resulting subset is  a good analogue of the augmented
base locus in codimension $1$.  

\begin{defn}\label{aug.bs.defn}
Let $D$ be a big $\mathbb R$-divisor on a normal proper
variety $X$. The \emph{divisorial augmented base locus} of $D$ is
the divisorial component $\bdp(D)$ of 
\begin{equation}\label{eq:bdpdef}\bigcap_{D=A+E}{\rm Supp}(E)\end{equation}
with the intersection being taken over all decompositions $D=A+E$ with
$A$ an $\mathbb R$-divisor, ample in codimension 1, and $E$ an effective $\mathbb R$-divisor. 

\end{defn}

The next lemma implies that if $X$ is projective then 
$\bdp(D)$ equals the divisorial part of ${\bf B}_{+}(D)$ and that
we can also compute $\bdp(D)$ in terms of just one divisor that is ample in codimension 1.

\begin{lem}\label{lem:bdpalternate}Let $X$ be a normal proper variety.  Let $D$ be a big $\mathbb{R}$-divisor and let $A$ be an $\mathbb R$-divisor that is ample in codimension 1 on $X$.
Then $\bdp(D)$ is the divisorial component of the intersection of the 
supports of all 
$D'\in |D-\epsilon A|_{\mathbb R}$ 
for all 
$\epsilon>0$.
\end{lem}

\begin{proof}
Let $U$ denote the above intersection.
Its index set is a subset of the one in \eqref{eq:bdpdef}, therefore $U\supseteq\bdp(D)$. 
Let now $\Gamma$ be a prime divisor which is a component of the supports
of all $D'\in|D-\epsilon A|_{\mathbb R}$ for all sufficiently small $\epsilon>0$.
Let $D=A'+E$ with $A'$ ample in codimension 1 and $E\geq 0$.
By Lemma \ref{concretesigma}.$i)$, for all sufficiently small $\epsilon>0$ there exists $B_{\epsilon}\sim_{\mathbb R}\epsilon A$ such that 
$A'-B_{\epsilon}\geq 0$ and 
$\Gamma\not\subset{\rm Supp}(A'-B_{\epsilon})$.
Then 
$$|D-\epsilon A|\ni D-B_{\epsilon}=(A'-B_{\epsilon})+E.$$
Consequently $\Gamma$ is a component of ${\rm Supp}(E)$.
\end{proof} 

The relationship between $\bdp(D)$ and the  Zariski decomposition
is given by the following.

\begin{lem}\label{lem:bdpzariski} Let $X$ be a normal proper variety.  Let $D$ be a big $\mathbb R$-divisor 
and  let $A$  be an $\mathbb R$-divisor which is ample in codimension 1 on $X$.
Then $$\bdp(D)={\rm Supp}(N_{\sigma}(D-\epsilon A))$$
for all sufficiently small $\epsilon>0$.
\end{lem}

\begin{proof}Note that since ${\rm Supp}(N_{\sigma}(D-\epsilon A))$ is a closed set, for any sufficiently small $\epsilon > 0$ the sets ${\rm Supp}(N_{\sigma}(D-\epsilon A))$ all coincide.  Thus we may show that $\bdp(D)$ coincides with the intersection over all sufficiently small $\epsilon > 0$ of the sets ${\rm Supp}(N_{\sigma}(D-\epsilon A))$.

By Theorem \ref{volprop}.$iii)$, we see that $D-\epsilon A$
is big for sufficiently small $\epsilon>0$.
Let $\Gamma$ be a prime divisor on $X$. Assume $\sigma_{\Gamma}(D-\epsilon A)=0$. Lemma \ref{concretesigma}.$ii)$ shows
that $\Gamma\not\subset{\rm Supp}(D')$ for some $D'\in|D-\frac{\epsilon}2A|_{\mathbb R}$. Therefore $\bdp(D)\subseteq \cap_{\epsilon > 0} {\rm Supp}(N_{\sigma}(D-\epsilon A))$. 
The reverse inclusion is straightforward from the previous lemma 
and the definition of $\sigma_{\Gamma}(D-\epsilon A)$.
\end{proof}

\begin{rmk}Inspired by \cite[Lemma 1.14]{elmnp06} we define the \emph{divisorial restricted base locus} as 
$$
{\bf B}^{\rm div}_{-}(D):=\bigcup_A\bdp(D+A),
$$
where $A$ ranges through all $\mathbb R$-divisors on $X$ that are
ample in codimension 1.
One can show that if the base field $K$ is uncountable and $D$ is a big $\mathbb R$-divisor, then
${\bf B}^{\rm div}_{-}(D)={\rm Supp}(N_{\sigma}(D))$.
\end{rmk}

\nocite{*}
\bibliographystyle{amsalpha}
\bibliography{Volumes}

\end{document}